\documentclass[preprint]{elsarticle}

\usepackage{amsmath,amsfonts}
\usepackage{enumitem,pinlabel,subfigure}
\newtheorem{theorem}{Theorem}[section]
\newtheorem{lemma}[theorem]{Lemma}
\newtheorem{corollary}[theorem]{Corollary}
\newtheorem{proposition}[theorem]{Proposition}
\newdefinition{example}{Example}
\newcommand{\ba}{\backslash}
\newcommand{\btu}{\bigtriangleup}
\newproof{proof}{Proof}
\newproof{prooft}{Proof of Theorem~2.4}

\begin{document}

\title{Inductive tools for connected delta-matroids and multimatroids\footnote{\copyright 2017
This manuscript version is made available under the CC-BY-NC-ND 4.0 license \texttt{http://creativecommons.org/licenses/by-nc-nd/4.0/}}\footnote{\texttt{http://dx.doi.org/10.1016/j.ejc.2017.02.005}}}

\author[a1]{Carolyn Chun}
\ead{chun@usna.edu}
\author[a2]{Deborah Chun}
\ead{deborah.chun@mail.wvu.edu}
\author[a3]{Steven D. Noble\corref{cor1}}
\ead{s.noble@bbk.ac.uk}
\cortext[cor1]{Corresponding author}
\address[a1]{Mathematics Department, United States Naval Academy, Chauvenet Hall, 572C Holloway Road, Annapolis, Maryland 21402-5002, United States of America}
\address[a2]{Department of Mathematics, West Virginia University Institute of Technology, Montgomery, West Virginia, United States of America}
\address[a3]{Department of Economic, Mathematics and Statistics, Birkbeck, University of London, Malet St, London, WC1E 7HX United Kingdom}

\begin{abstract}
We prove a splitter theorem for tight multimatroids, generalizing the corresponding result for matroids, obtained independently by Brylawski and Seymour. Further corollaries give splitter theorems for delta-matroids and ribbon graphs.

\end{abstract}

\begin{keyword}matroid \sep delta-matroid \sep partial dual \sep minors \sep inductive tool \sep chain theorem \sep splitter theorem \sep $2$-connected \sep connected delta-matroid
\MSC[2014]{05B35}
\end{keyword}
\date{\today}

\maketitle

\section{Introduction}
A \emph{matroid} $M=(E,\mathcal{B})$ is a finite \emph{ground set} $E$ together with a non-empty collection of subsets of the ground set, $\mathcal{B}$, that are called \emph{bases}, satisfying the following conditions, which are stated in a slightly different way from what is most common in order to emphasize the connection with other combinatorial structures discussed in this paper.
\begin{enumerate}
\item If $B_1$ and $B_2$ are bases and $x\in B_1 \btu B_2$, then there exists $y \in B_1\btu B_2$ such that $B_1 \btu\{x,y\}$ is a basis.
\item All bases are equicardinal.
\end{enumerate}
Matroid theory is often thought of as a generalization of graph theory, as a matroid $(M,\mathcal B)$ may be constructed from a graph $G$ by taking $E$ to be set of edges of $G$ and $\mathcal B$ to be the edge sets of maximal spanning forests of $G$.
Graph theory and matroid theory are mutually enriching: many results in graph theory have been generalized to matroids, and results in matroid theory have sometimes been proved before the corresponding specialization in graph theory. In~\cite{delta-ribbon}, Chun, Moffatt, Noble and Rueckriemen showed that the mutually-enriching relationship between graphs and matroids is analogous to the mutually-enriching relationship between cellularly-embedded graphs, which we view as ribbon graphs, and objects called \emph{delta-matroids}. They gave further evidence for this by establishing several new results for delta-matroids in~\cite{delta-twist}, each of which was inspired by a previously known result concerning ribbon graphs.

Delta-matroids were extensively studied by Bouchet in the 1980s, but until recently had been little studied since that foundational work. In addition to~\cite{delta-twist,delta-ribbon}, where the authors were led to delta-matroids by studying ribbon graphs, they have been studied extensively by Brijder and Hoogeboom who were originally interested in the principal pivot transform in binary matrices (see, for example,~\cite{BH11,BH13,BH14}).

A \emph{delta-matroid} $D=(E,\mathcal{F})$ is a finite \emph{ground set} $E$ together with a non-empty collection of subsets of the ground set, $\mathcal{F}$, that are called \emph{feasible sets}, satisfying the following condition known as the \emph{symmetric exchange axiom}.
If $F_1$ and $F_2$ are feasible sets and $x\in F_1 \btu F_2$, then there exists $y \in F_1\btu F_2$ such that $F_1 \btu\{x,y\}$ is a feasible set. Note that we allow $y=x$.
It follows immediately from the definitions that every matroid is a delta-matroid.
In fact, the axiom for the feasible sets of a delta-matroid corresponds exactly to (1) in the axioms we gave earlier for the bases of a matroid. A delta-matroid is said to be \emph{even} if the sizes of its feasible sets all have the same parity. Thus a matroid is an even delta-matroid.

As in many other areas of mathematics, structural results on matroids often require an assumption of some level of connectivity of the matroid.
In~\cite{geelen}, Geelen defined connectivity for delta-matroids as follows.
Given delta-matroids $D_1=(E_1,\mathcal F_1)$ and $D_2=(E_2, \mathcal F_2)$ with disjoint ground sets, their \emph{direct sum}, written $D_1 \oplus D_2$, is the delta-matroid with ground set $E_1\cup E_2$ and collection of feasible sets $\{F_1 \cup F_2: F_1\in \mathcal F_1 \text{ and } F_2\in\mathcal F_2\}$. If $D=D_1\oplus D_2$ then we say that $E(D_1)$ and $E(D_2)$ are \emph{separators} of $D$. If $X$ is a separator of a delta-matroid $D$ and $\emptyset \ne X \ne E(D)$ then we say that $X$ is a \emph{proper separator} of $D$.
A delta-matroid $D$ is \emph{disconnected} if it has a proper separator. Otherwise $D$ is \emph{connected}.
Clearly the matroids that satisfy the definition of delta-matroid connectivity are exactly those that satisfy the well-known definition of matroid connectivity~\cite{Oxley11}.
Moreover when applied to matroids, the definition of a separator in a delta-matroid is exactly the same as that of a separator in a matroid~\cite{Oxley11}.
Our aim is to study the effect on connectivity of removing elements from a delta-matroid. As a consequence we provide useful tools for inductive proofs of results concerning $2$-connected ribbon graphs, which we define later.

Deletion and contraction are the two natural ways in which to remove an element from a matroid or delta-matroid.
For a delta-matroid $D=(E,\mathcal{F})$, and $e\in E$, if $e$ is in every feasible set of $D$, then we say that $e$ is a \emph{coloop of $D$}.
If $e$ is in no feasible set of $D$, then we say that $e$ is a \emph{loop of $D$}.
If $e$ is not a coloop, then, following Bouchet and Duchamp~\cite{BD91}, we define $D$ \emph{delete} $e$, written $D\ba e$, to be
\[D\ba e=(E-e, \{F : F\in \mathcal{F}\text{ and } F\subseteq E-e\}).\]
If $e$ is not a loop, then we define $D$ \emph{contract} $e$, written $D/e$, to be \[D/e=(E-e, \{F-e : F\in \mathcal{F}\text{ and } e\in F\}).\]
If $e$ is a loop or coloop, then $D/e=D\ba e$.

Both $D\ba e$ and $D/e$ are delta-matroids (see \cite{BD91}).
Let $D'$ be a delta-matroid obtained from $D$ by a sequence of deletions and contractions.  Then $D'$ is independent of the order of the  deletions and contractions used in its construction (see \cite{BD91}) and $D'$ is called a \emph{minor} of $D$.
We let $D|A$ denote $D\ba (E-A)$. All of these definitions are entirely consistent with the corresponding better-known definitions for matroids.

Two early results describing the effect of deleting or contracting an element from a matroid are the following. The first was proved by Tutte~\cite{Tutte66} and the second independently by Brylawski~\cite{Bry72} and Seymour~\cite{Seymour77}.
\begin{theorem}\label{th:tutte}
Let $e$ be an element of a connected matroid $M$. Then either $M\ba e$ or $M/e$ is connected.
\end{theorem}

\begin{theorem}\label{th:bry}
Let $N$ be a connected minor of a connected matroid $M$ and let $e$ be an element of $E(M)-E(N)$. Then either $M/e$ or $M\ba e$ is connected and has $N$ as a minor.
\end{theorem}

Results of the first type are known as chain theorems; results of the second type are known as splitter theorems. Our original aim was to prove a splitter theorem for connected even delta-matroids, but it turns out that the natural setting for these results is an even more general object, namely multimatroids, which we discuss in the next section. Working in this more general setting requires no extra effort and indeed allows 
us to make use of previous work of Bouchet establishing a chain theorem for connected multimatroids~\cite[Theorem 8.7]{mmiii}. As we shall see later, Bouchet noted that this result implied a chain theorem for even delta-matroids.

The structure of this paper is as follows. In the next section we describe multimatroids and prove our main result; in the final section we describe the implications of this result to delta-matroids and ribbon graphs.

\section{Multimatroids and the main result}
We begin by defining a multimatroid and associated terminology. All definitions follow Bouchet~\cite{mmi,mmii,mmiii}.
Let $U$ be a finite set and $\Omega$ a partition of $U$, where each set of the partition is called a \emph{skew class}.
Every pair of elements contained in a skew class is a \emph{skew pair}.
A set $T\subseteq U$ is a \emph{transversal} of $\Omega$ if it meets each skew class in exactly one element, and a set is a \emph{subtransversal} of $\Omega$ if it is contained in a transversal of $\Omega$.
Let $\mathcal{S}(\Omega)$ be the set of subtransversals of $\Omega$.
The triple $Q=(U,\Omega ,r)$ is a \emph{multimatroid}, where $r:\mathcal{S}(\Omega)\rightarrow \mathbb{Z}^+$ is its \emph{rank function}, if $r$ obeys the following axioms:
\begin{enumerate}
\item $r(\emptyset )=0$;
\item $r(A)\leq r(A\cup x)\leq r(A)+1$, if $A\in\mathcal{S}(\Omega)$ and $x$ is an element in a skew class that avoids $A$;
\item $r(A)+r(B)\geq r(A\cup B)+r(A\cap B)$, if $A\cup B$ is in $\mathcal{S}(\Omega)$; and
\item $r(A\cup x)-r(A)+r(A\cup y)-r(A)\geq 1$, if $A\in\mathcal{S}(\Omega)$ and $\{x,y\}$ is a skew pair in a skew class that avoids $A$.
\end{enumerate}
A multimatroid whose skew classes each have size $q$ is called a \emph{$q$-matroid}. It follows immediately from the definition that $(U,\Omega,r)$ is a $1$-matroid if and only if it is a matroid with ground set $U$ and rank function $r$. We will see in the next section that there is a correspondence between $2$-matroids and delta-matroids.

A subtransversal is an \emph{independent set} if its rank is equal to its cardinality, otherwise it is \emph{dependent}.
The maximal independent sets are the \emph{bases} of a multimatroid.
If no skew class consists of a single element, then the multimatroid is \emph{non-degenerate}, and Bouchet~\cite[Proposition~5.5]{mmi} showed that the bases of a non-degenerate multimatroid are transversal.
A subtransversal is a \emph{circuit} if it is dependent but every proper subset is independent.

Let $Q=(U,\Omega ,r)$ be a multimatroid and take $A\in\mathcal{S}(\Omega)$.
Let $\Omega '=\{\omega \in\Omega :\omega \cap A=\emptyset\}$, let $U'\subseteq U$ be the set of elements in the skew classes of $\Omega '$ and let $r':\mathcal{S}(\Omega ')\rightarrow \mathbb{Z}^+$ be defined by
\begin{equation} r'(X)=r(X\cup A)-r(A).\label{eq:minorrank}\end{equation}
Then it is straightforward to verify that $(U',\Omega ',r')$ is a multimatroid which
we call the \emph{minor of $Q$ with respect to $A$} and which we write as $Q|A$.
More generally, we say that $(U',\Omega ',r')$ is a \emph{minor} of $Q$.
It follows immediately from \eqref{eq:minorrank} that if $A$ and $B$ are disjoint and such that $A \cup B \in \mathcal{S}(\Omega)$, then $(Q|A)|B = Q|A\cup B = (Q|B)|A$.

An element in a multimatroid is \emph{singular} if it has rank zero.
A skew class is \emph{singular} if it contains a singular element.
The following lemma of Bouchet~\cite[Proposition~5.5]{mmii} is needed.
\begin{lemma}
\label{sing}
Let $\omega$ be a skew class of a multimatroid $Q$.
If $\omega$ is singular, then, for every pair of elements $\{e,f\}\subseteq \omega$, the minors $Q|e$ and $Q|f$ are equal.
\end{lemma}

The following is a slight generalization of a theorem of Bouchet~\cite[Theorem~5.6]{mmii} and is similar to the Scum Theorem in matroid theory. The proof is a straightforward extension of Bouchet's but is included for completeness.
\begin{theorem}
\label{scum}
For a non-degenerate multimatroid $Q=(U,\Omega ,r)$, $A\in\mathcal{S}(\Omega)$ and element $e$ of $A$ satisfying $r(e)=1$, there is an independent set $I$ of $Q$ such that $e\in I$ and $Q|A=Q|I$.
\end{theorem}
\begin{proof}
We proceed by induction on $|A|$. If $|A|=1$, then the result is clear. Otherwise choose an element $x$ other than $e$ in $A$. Let $Q'=Q|x$ and $A'=A-x$. By induction there is an independent set $I'$ of $Q'$ such that $e\in I'$ and $Q'|I'=Q'|A'$. If $I' \cup x$ is an independent set of $Q$ then the proof is complete. So we may assume that $I' \cup x$ is dependent in $Q$ and thus $r(I' \cup x) \leq |I'|$. Since $I'$ is independent in $Q'$, we have $|I'| = r(I' \cup x)-r(x)$. Consequently $r(x)=0$ and so the skew class containing $x$ is singular in $Q$. Choose another element $y$ from this skew class. Then by Axiom~(4) in the definition of a multimatroid, $r(y)=1$, and by Lemma~\ref{sing}, $Q|x=Q|y$. Now choose $I=I' \cup y$. We have $Q|A=Q|x|A'=Q'|A'=Q'|I'=Q|y|I'=Q|I$ and $r(I)=r(I' \cup y) = r'(I')+r(y)=|I'|+1=|I|$, where $r'$ denotes the rank function of $Q'$. Hence the result follows by induction.\qed
\end{proof}

A set $X\subseteq U$ is a \emph{separator} of $Q$ if $X$ is a union of skew classes of $\Omega$ such that, for all $A\in\mathcal{S}(\Omega )$,
\[r(A)=r(A\cap X)+r(A-X).\]
We say that a separator $X$ is \emph{proper} if $X$ is non-empty and $X\ne U$.
A multimatroid $Q$ is \emph{disconnected} if it has a proper separator. Otherwise $Q$ is \emph{connected}.
Notice that separators of a $1$-matroid are precisely the separators of the corresponding matroid and that a $1$-matroid is connected if and only if the corresponding matroid is connected.

We will restrict our attention to tight multimatroids. We shall see later that tight $2$-matroids correspond to the class of even delta-matroids and that tight $3$-matroids correspond to the class of vf-safe delta-matroids, which we define later.
Let $Q=(U,\Omega ,r)$ be a multimatroid.
We say that a subtransversal is a \emph{near-transversal} if it meets all of the skew classes except for one.
Then $Q$ is \emph{tight} if it is non-degenerate and for every skew class $\omega$ and every near-transversal $A$ that avoids $\omega$,
\[\sum _{x\in\omega} (r(A\cup x)-r(A))= |\omega | -1.\]
By Axiom~(4) for the multimatroid rank function, the left-hand side is bounded below by the right-hand side for all multimatroids, but we insist on equality in the case of a tight multimatroid.
Bouchet~\cite[Proposition~4.1]{mmiii} showed that every minor of a tight multimatroid is tight.
The main result in~\cite{mmiii} is the following chain theorem by Bouchet.
\begin{theorem}
\label{chainmm}
Let $\{e_1,e_2,\dots ,e_k\}$ be a skew class of a connected tight multimatroid $Q$.
At least $k-1$ of the minors in $\{Q|e_1,Q|e_2,\dots ,Q|e_k\}$ are connected.
\end{theorem}
Bouchet~\cite{mmiii} provided an example, which is attributed to an unpublished manuscript of Gasse, showing that the tightness condition is necessary.

The following splitter theorem is our main result.
\begin{theorem}
\label{splitmm}
Let $Q$ be a connected tight multimatroid and let $A$ be a non-empty subtransversal such that $Q|A$ is connected.
If $e\in A$, then
\begin{enumerate}[label=(\roman*)]
\item $Q|e$ is connected; or
\item
for all $x$ such that $\{e,x\}$ is a skew pair, $Q|x$ is connected with $Q|A$ as a minor.
\end{enumerate}
\end{theorem}

The remainder of this section is devoted to proving this result. A key notion in the proof is that of a fundamental circuit which generalizes the notion of a fundamental circuit of a matroid. Let $B$ be a basis and $\omega$ be a skew class of a non-degenerate multimatroid $Q$. Then it follows immediately from the definition of a multimatroid that $B\cup \omega$ contains at most one circuit. Furthermore, if $Q$ is tight, then $B\cup \omega$ contains precisely one circuit. Following Bouchet~\cite{mmiii}, this circuit is called the \emph{fundamental circuit} of $Q$ with respect to $B$ and $\omega$, and is denoted by $C(B, \omega)$. Define a relation $\sim_B$ on the elements of $B$, by $e\sim_B f$ if $e$ belongs to the fundamental circuit of $Q$ with respect to $B$ and the skew class containing $f$.
Bouchet~\cite[Proposition 6.1]{mmiii} showed that $\sim_B$ is symmetric. The graph of $\sim_B$ is called the \emph{fundamental graph} of $B$. The following theorem, combining a special case of Proposition~7.3 and Theorem~8.3 from~\cite{mmiii}, describes the properties of fundamental graphs that we will need.

\begin{theorem}\label{th:fundgr}
Let $Q$ be a tight multimatroid, $B$ a basis of $Q$ and $G$ the fundamental graph of $B$. Then the following hold.
\begin{enumerate}[label=(\roman*)]
\item If $e\in B$ then $B-e$ is a basis of $Q|e$ and its fundamental graph is obtained from $G$ by deleting $e$ and all of its incident edges.
\item The fundamental graph $G$ is connected if and only if $Q$ is connected. Moreover $X$ is a separator of $Q$ if and only if $X$ is formed by choosing a (possibly empty) collection of connected components of $G$ and taking the union of all the skew classes corresponding to elements of $B$ belonging to these connected components.
\end{enumerate}
\end{theorem}

We also need the following lemma due to Bouchet~\cite[Lemma 8.5]{mmii}.
\begin{lemma}
\label{rank1}
If a multimatroid $(U,\Omega ,r)$ is connected and has more than one skew class, then $r(e)=1$ for all $e\in U$.
\end{lemma}

Combining the previous results enables us to find a circuit with particularly useful properties.

\begin{lemma}
\label{Xcct}
Let $Q$ be a connected tight multimatroid containing an element $e$ such that $Q|e$ is disconnected. If $X$ is a proper separator of $Q|e$ then $Q$ has a circuit $C$ such that $e \in C \subseteq X \cup e$.
\end{lemma}
\begin{proof}
Lemma~\ref{rank1} implies that $r(e)=1$, hence $e$ is contained in a basis $B$ of $Q$.
Theorem~\ref{th:fundgr} implies that the fundamental graph $G$ of $B$ is connected and that deleting $e$ from $G$ gives a disconnected graph.
So $G-e$ is disconnected but each connected component of $G-e$ has at least one vertex that is adjacent to $e$ in $G$.
Let $X$ be a proper separator of $Q|e$. Then $X$ is the union of all the skew classes corresponding to elements of $B-e$ belonging to at least one but not all of the connected components of $G-e$.
There is an element $f \in B \cap X$ such that $f$ is adjacent to $e$ in $G$. Let $C$ be the fundamental circuit of $Q$ with respect to $B$ and the skew class containing $f$.
Then $C$ is a circuit of $Q$. It contains $e$ by the definition of the edges of the fundamental graph. Moreover, this circuit does not contain any element of $B-e-X$, again by the definition of the edges of the fundamental graph and the connectivity properties of $G$ and $G-e$. Thus $e\in C \subseteq X \cup e$ and the lemma holds. \qed
\end{proof}

The proof of the following lemma requires applying the definition of a separator and the rank function of a minor with some straightforward manipulation and is omitted.
\begin{lemma}
\label{minorsep}
Let $X$ be a separator in a multimatroid $Q$ and let $A$ be a subtransversal of $Q$.
Let $U_A$ be the union of the skew classes of $Q$ that meet $A$.
Then $X-U_A$ is a separator in $Q|A$.
\end{lemma}

Next we see that whenever we take a minor with respect to a sub-transversal of some skew classes forming a separator, it does not matter which subtransversal we choose to form the minor and the resulting multimatroid has a simple description.
\begin{lemma}\label{lem:ref}
Let $Q=(U,\Omega,r)$ be a multimatroid, $X$ be a separator of $Q$ and $A$ be a subtransversal such that $A \subseteq X$ and $A$ meets every skew class included in $X$. Then $Q|A$ is the multimatroid with ground set $U-X$, having as skew classes the skew classes of $Q$ avoiding $X$ and as rank function the restriction of $r$ to subtransversals of $U-X$.
\end{lemma}

\begin{proof}
We must check that the rank function of $Q|A$ is as described. Let $S$ be a subtransversal of the skew classes of $Q|A$. Then
\[ r_{Q|A}(S) = r(S\cup A) - r(A) = r((S\cup A) \cap X) + r((S\cup A)-X) - r(A).\]
However $(S\cup A) \cap X = A$ and $(S \cup A)- X = S$. Thus $r_{Q|A}(S) = r(S)$ as required. \qed
\end{proof}

%

We are now in a position to prove our main result.
\begin{prooft}
Suppose that (i) does not hold.

By Lemma~\ref{rank1}, $\{e\}$ is independent in $Q$.
By Theorem~\ref{scum}, we may assume that $A$ is independent in $Q$.

Now $Q|e$ has a proper separator $X$.
Let $Y$ be the complement of $X$ in $Q|e$.
As $Q|A$ has no separator, Lemma~\ref{minorsep} implies that the elements in $Q|A$ are all contained in $X$ or all contained in $Y$.
Without loss of generality, since both $X$ and $Y$ are separators in $Q|e$, we assume that the elements of $Q|A$ are contained in $Y$.

By Lemma~\ref{Xcct}, we know that $Q$ has a circuit $C$ such that $e\in C$ and $C\subseteq X\cup e$.
Let $Z$ be a subtransversal of $Q|e$ containing $C-e$ and meeting every skew class in $X$, and let $A'$ be the restriction of $A$ to the skew classes in $X$. Then Lemma~\ref{lem:ref} implies that $Q|e|A'=Q|e|Z$. Hence $Q|A$ is a minor of $Q|e|Z$ which is a minor of $Q|C$.
%
%

As $C$ is a circuit in $Q$, the rank $r_{Q|(C-e)}(e)=r_{Q}(C)-r_{Q}(C-e)=0$.
Hence $e$ is singular in $Q|(C-e)$.
Lemma~\ref{sing} implies that $Q|C=( Q|(C-e))|x=(Q|x)|(C-e) $ for all $x$ in the skew class containing $e$.
Theorem~\ref{chainmm} implies that (ii) holds. \qed
\end{prooft}

Notice that if case (i) of Theorem~\ref{splitmm} does not hold, then $Q|x$ is connected and contains $Q|A$ as a minor for every $x$ in the skew class containing $e$ except for $e$. In contrast, if case (i) holds, then it is possible that $Q|A$ is not a minor of $Q|x$ for any $x$ in the skew class of $e$ except $e$ itself. The following example illustrates this.

\begin{example}
Let $Q$ be the multimatroid with skew classes $\{a,a',a''\}$, $\{b,b',b''\}$, $\{c,c',c''\}$ and $\{d,d',d''\}$, and bases as shown in Table~\ref{tab:multi}.
In the next section we will describe a correspondence due to Brijder and Hoogeboom~\cite{BH14} between certain delta-matroids and tight $3$-matroids.
In this case $Q$ is constructed from the delta-matroid with ground set $\{a,b,c,d\}$ and collection of feasible sets
\[ \mathcal{F} = \{\{\emptyset\},\{a\},\{b\},\{c\},\{d\},\{a,b\},\{c,d\},\{a,b,c\},\{a,b,d\},\{a,c,d\},\{b,c,d\}\}.\]
To verify that $\mathcal F$ is the collection of feasible sets of a delta-matroid, we must verify that the symmetric exchange axiom holds.
Because $\mathcal{F}$ contains every feasible set with odd size, it follows that whenever $F$ has even size, $F\btu e \in \mathcal{F}$ for every $e\in\{a,b,c,d\}$. Due to symmetry, it remains to show that the symmetric exchange axiom holds when $F_1 = \{a\}$ or $F_1=\{a,b,c\}$. We may assume that $|F_1 \btu F_2| \geq 3$. Thus the only remaining pairs of sets for which the symmetric exchange axiom must be verified are given by
\[(F_1,F_2) \in \{(\{a\},\{c,d\}),(\{a\},\{b,c,d\}),(\{a,b,c\},\{d\}),(\{a,b,c\},\emptyset),(\{a,b,c\},\{c,d\})\}.\]
Each of these cases is easily checked.

Both $\{a,b,c''\}$ and $\{a,b,d''\}$ are circuits of $Q$, so the fundamental graph of $Q$ with respect to the basis $\{a,b,c,d\}$ is connected. Consequently it follows from Theorem~\ref{th:fundgr} that $Q$ is connected.

Now consider $Q|a'$.
Neither $Q|a$ nor $Q|a''$ contain $Q|a'$ as a minor, because $Q|a'$ has more bases than the other two. Moreover $Q|a'$ is connected, because $\{b,c,d'\}$ is one of its circuits.

Note that in this example something slightly stronger holds: neither $Q|a$ nor $Q|a''$ is isomorphic to $Q|a'$. There are connected tight $3$-matroids with three skew classes containing an element $a$ such that $Q|a$ is connected but for any $x$ other than $a$ in the skew class containing $a$, $Q|x$ does not contain $Q|a$ as minor. However in all these cases $Q|x$ is isomorphic to $Q|a$ whenever $Q|x$ is connected. Consequently $Q$ is the smallest example for which this stronger property holds.
\end{example}

\begin{table}[ht]\begin{center}
\begin{tabular}{ccc}
$\{a,b,c,d\}$&$\{a',b,c,d\}$&$\{a'',b,c,d'\}$\\
$\{a,b,c,d'\}$&$\{a',b,c,d''\}$&$\{a'',b,c,d''\}$\\
$\{a,b,c',d\}$&$\{a',b,c',d'\}$&$\{a'',b,c',d\}$\\
$\{a,b,c',d'\}$&$\{a',b,c',d''\}$&$\{a'',b,c',d''\}$\\
$\{a,b',c,d\}$&$\{a',b,c'',d\}$&$\{a'',b,c'',d\}$\\
$\{a,b',c,d''\}$&$\{a',b,c'',d'\}$&$\{a'',b,c'',d'\}$\\
$\{a,b',c',d'\}$&$\{a',b',c,d\}$&$\{a'',b',c,d'\}$\\
$\{a,b',c',d''\}$&$\{a',b',c,d'\}$&$\{a'',b',c,d''\}$\\
$\{a,b',c'',d\}$&$\{a',b',c',d\}$&$\{a'',b',c',d\}$\\
$\{a,b',c'',d'\}$&$\{a',b',c',d''\}$&$\{a'',b',c',d'\}$\\
$\{a,b'',c,d'\}$&$\{a',b',c'',d'\}$&$\{a'',b',c'',d\}$\\
$\{a,b'',c,d''\}$&$\{a',b',c'',d''\}$&$\{a'',b',c'',d''\}$\\
$\{a,b'',c',d\}$&$\{a',b'',c,d'\}$&$\{a'',b'',c',d'\}$\\
$\{a,b'',c',d''\}$&$\{a',b'',c,d''\}$&$\{a'',b'',c',d''\}$\\
$\{a,b'',c'',d\}$&$\{a',b'',c',d\}$&$\{a'',b'',c'',d'\}$\\
$\{a,b'',c'',d'\}$&$\{a',b'',c',d'\}$&$\{a'',b'',c'',d''\}$\\
&$\{a',b'',c'',d\}$&\\
&$\{a',b'',c'',d''\}$&
\end{tabular}
\end{center}
\caption{Bases of the multimatroid $Q$}
\label{tab:multi}
\end{table}

\section{Applications to delta-matroids and ribbon graphs}

We begin by briefly describing the relationship between delta-matroids and $2$-matroids from~\cite{mmi}.
Bouchet notes in~\cite{mmi} that a $2$-matroid is determined by its bases, proving the following.
\begin{theorem}\label{th:Bouchetbases}
Let $U$ be a finite set and $\Omega$ be a partition of $U$ into pairs. Then a non-empty collection $\mathcal B$ of transversals of $\Omega$ is the collection of bases of a $2$-matroid if and only if whenever $B_1$ and $B_2$ belong to $\mathcal B$ and $p$ is a skew pair such that $p \subseteq B_1 \btu B_2$, there is a skew pair $q$ such that $B_1 \btu (p \cup q) \in \mathcal B$.
\end{theorem}
Let $D=(E,\mathcal F)$ be a delta-matroid. Now we construct a $2$-matroid $Q_2(D)$ as follows. The ground set
is $U=\{e,e':e\in E\}$. The set of skew classes is $\Omega=\{\{e,e'\}:e\in E\}$. For a subset $A$ of $E$, we define $A'=\{e':e\in E\}$. Then $Q_2(D)$ has a basis $F \cup (E-F)'$ corresponding to each feasible set $F$ of $D$. It follows from Theorem~\ref{th:Bouchetbases} that $Q_2(D)$ is indeed a $2$-matroid. On the other hand suppose that $Q=(U,\Omega,r)$ is a $2$-matroid, $\mathcal B$ is its collection of bases and $T$ is a transversal of $\Omega$. Then the \emph{section} of $Q$ by $T$ is a delta matroid with ground set $T$ and set of feasible sets equal to $\{B\cap T:B\in\mathcal{B}\}$. Using Theorem~\ref{th:Bouchetbases}, one may verify that a section is indeed a delta-matroid. In~\cite{mmiii}, Bouchet proves that $Q_2(D)$ is tight if and only if $D$ is even and, conversely, that every section of $Q$ is even if and only if $Q$ is tight. Note that if one section of $Q$ is even then all sections of $Q$ are even.

It is not difficult to check that if $e$ is an element of a delta-matroid $D$, then $Q_2(D/e)=Q_2(D)|e$ and $Q_2(D\ba e)=Q_2(D)|e'$. Furthermore one may also define a \emph{direct-sum} for multimatroids. Let $Q_1$ and $Q_2$ be multimatroids on disjoint ground sets $U_1$ and $U_2$, sets of skew classes $\Omega_1$ and $\Omega_2$ and sets of bases $\mathcal B_1$ and $\mathcal B_2$ respectively. Then
$Q_1 \oplus Q_2$ is the multimatroid with ground set $U_1 \cup U_2$, set of skew classes $\Omega_1 \cup \Omega_2$ and set of bases $\{B_1\cup B_2: B_1 \in \mathcal{B}_1 \text{ and } B_2 \in \mathcal{B}_2\}$. Now it is easy to see that $Q$ fails to be connected if and only if $Q=Q_1 \oplus Q_2$ for two multimatroids $Q_1$ and $Q_2$, each of which has a non-empty ground set. It follows from this that $Q_2(D)$ is connected if and only if $D$ is connected and, conversely, that every section of $Q$ is connected if and only if $Q$ is connected. Again, note that if one section of $Q$ is connected, then all sections of $Q$ are connected.

Consequently all the key notions in delta-matroids and $2$-matroids correspond and we may deduce the following from Theorem~\ref{chainmm} and Theorem~\ref{splitmm}, respectively.

\begin{corollary}
\label{chain}
Let $D$ be a connected even delta-matroid.
If $e\in E(D)$, then $D\ba e$ or $D/e$ is connected.
\end{corollary}

\begin{corollary}
\label{splitter}
Let $D$ be a connected even delta-matroid with a connected minor $D'$.
If $e\in E(D)-E(D')$, then $D\ba e$ or $D/e$ is connected with $D'$ as a minor.
\end{corollary}

Because every matroid is an even delta-matroid, we also immediately obtain Theorems~\ref{th:tutte} and~\ref{th:bry} as corollaries.
Furthermore, the example that Bouchet gave in~\cite{mmiii} to show that the chain theorem for connected tight multimatroids does not hold for connected multimatroids in general is a $2$-matroid.
Hence this example also shows that Corollary~\ref{chain} does not hold for connected delta-matroids in general.

Ribbon graphs provide an alternative description of cellularly embedded graphs that is more natural for the present setting.
A \emph{ribbon graph} $G =\left(  V(G),E(G)  \right)$ is a surface with boundary, represented as the union of two  sets of  discs: a set $V (G)$ of \emph{vertices} and a set of \emph{edges} $E (G)$ with the following properties.
\begin{enumerate}
 \item The vertices and edges intersect in disjoint line segments.
 \item Each such line segment lies on the boundary of precisely one vertex and precisely one edge.
 \item Every edge contains exactly two such line segments.
\end{enumerate}

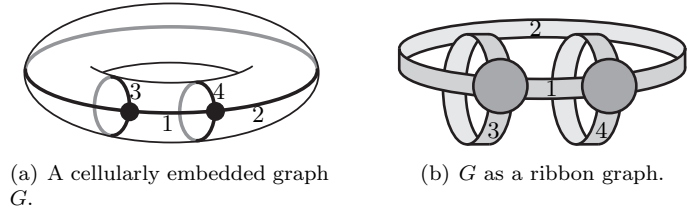
\begin{figure}
\centering
\subfigure[A cellularly embedded graph $G$.]{
\labellist \small\hair 2pt
\pinlabel {$1$} at   84 14
\pinlabel {$2$} at    135 19
\pinlabel {$3$} at    66 32
\pinlabel {$4$} at   113 32
\endlabellist
\raisebox{0mm}{\includegraphics[height=2cm]{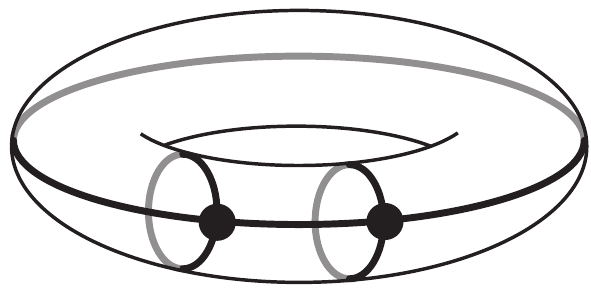}}
 \label{f.desca}
}
\hspace{5mm}
\subfigure[$G$ as a ribbon graph.]{
\labellist \small\hair 2pt
\pinlabel {$1$} at   60 22.7
\pinlabel {$2$} at   54 45.6
\pinlabel {$3$} at   38 7
\pinlabel {$4$} at   79 7
\endlabellist
\includegraphics[height=2cm]{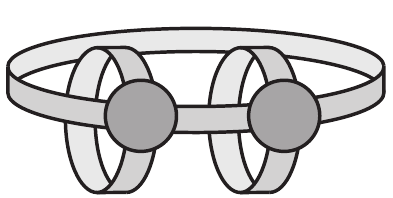}
 \label{f.descb}
}

\caption{Embedded graphs and ribbon graphs.}
\label{f.desc}
\end{figure}

It is well-known that ribbon graphs are just descriptions of cellularly-embedded graphs (see for example~\cite{GT87}).
We say that two ribbon graphs are \emph{equivalent} if they define equivalent cellularly embedded graphs, and we consider ribbon graphs up to equivalence. This means that ribbon graphs are considered up to homeomorphisms that preserve the graph structure  of the ribbon graph and the cyclic order of half-edges at each of its vertices.
We say that a ribbon graph is \emph{orientable} if it is orientable when regarded as a surface with boundary. A loop in a ribbon graph is \emph{orientable} if the subgraph comprising the loop and the vertex it meets is an orientable ribbon graph.


Let $G=(V,E)$ be a ribbon graph. If $e$ is an edge of a ribbon graph $G$, then \emph{edge deletion} is defined by $G\ba e= (V, E-e)$.
The definition of \emph{edge contraction} $G/e$ is a little more involved. For the purposes of this paper, we define it merely by illustrating its effect on different types of edges as shown in Table~\ref{tablecontractrg}. For a formal definition, see~\cite{delta-ribbon,EMMbook}. It is not too difficult to show that the definitions may be extended to deleting or contracting sets of edges. If some edges in a ribbon graph are selected for deletion and some others are selected for contraction, then the same ribbon graph will be produced regardless of the order of operations.
Again, for full details, see~\cite{delta-ribbon,EMMbook}.
If $H$ is obtained from a ribbon graph $G$ by a sequence of edge deletions, vertex deletions, and edge contractions, then we say that $H$ is a \emph{minor} of $G$.

\begin{table}
\centering
\begin{tabular}{|c||c|c|c|}\hline
 &  non-loop & non-orientable loop&orientable loop\\ \hline
\raisebox{6mm}{$G$} &
\includegraphics[scale=.25]{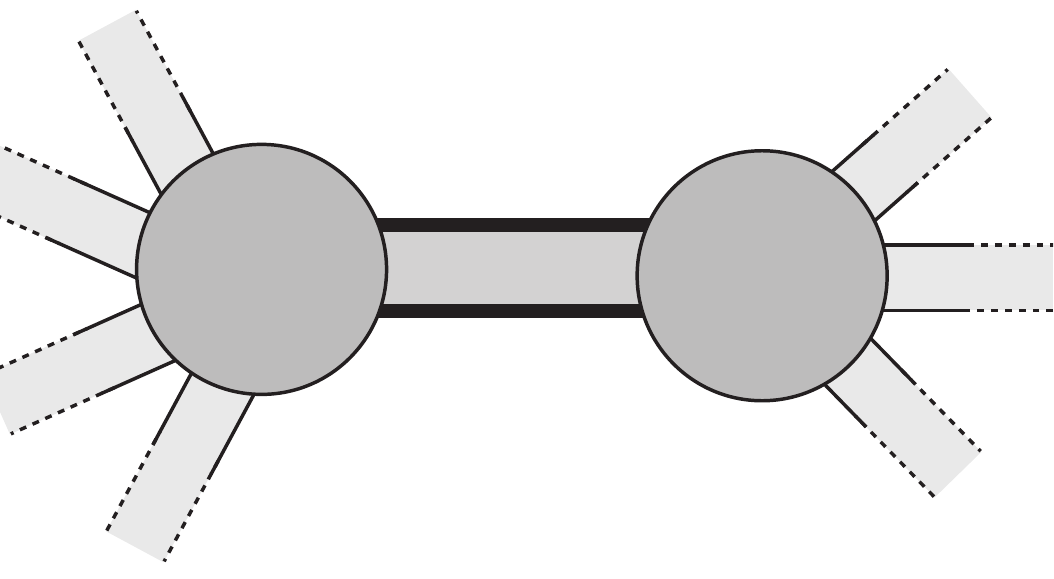} &\includegraphics[scale=.25]{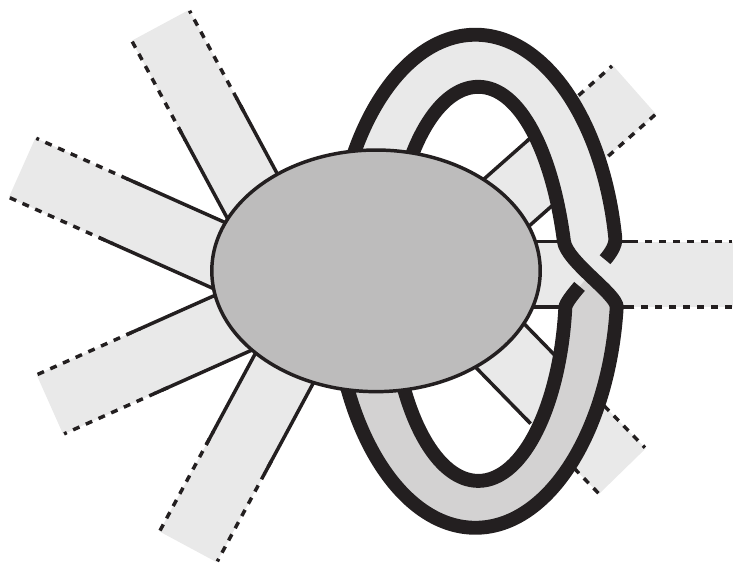} &\includegraphics[scale=.25]{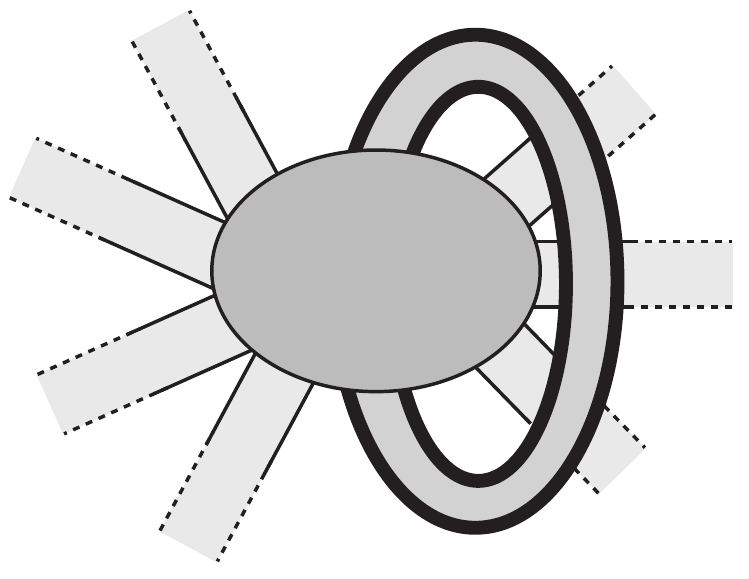}
\\ \hline
\raisebox{6mm}{$G/e$} &
\includegraphics[scale=.25]{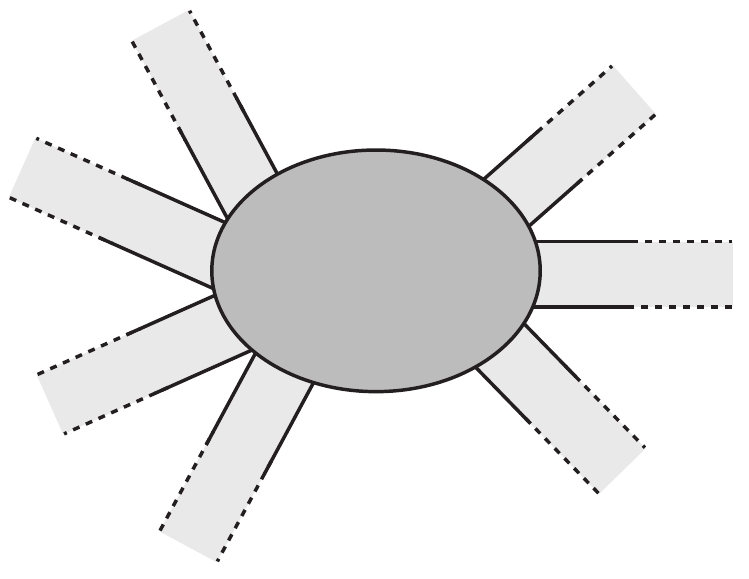} &\includegraphics[scale=.25]{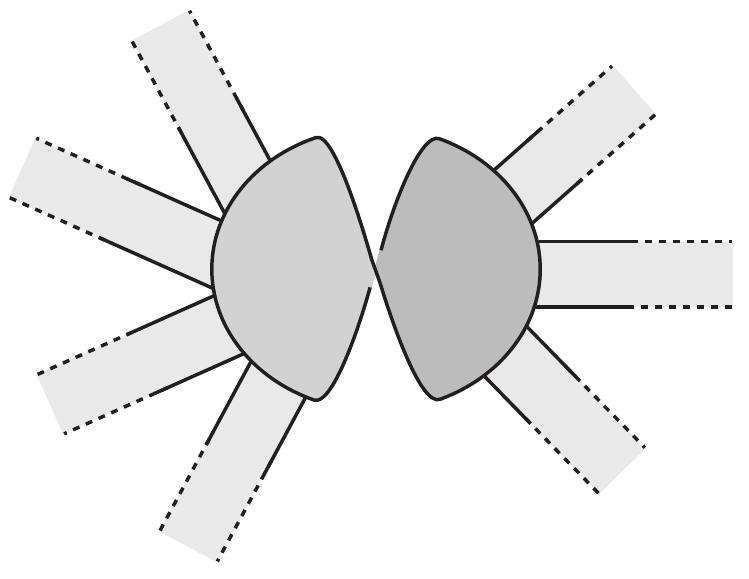}&\includegraphics[scale=.25]{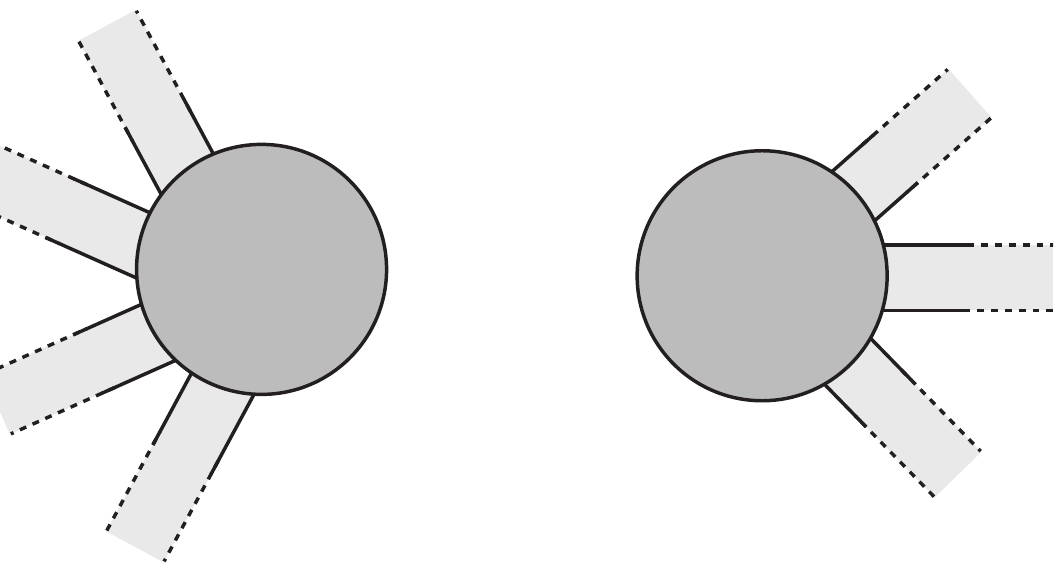} \\\hline
\raisebox{6mm}{$G*{e}$} &
\includegraphics[scale=.25]{ch4_35.pdf} &\includegraphics[scale=.25]{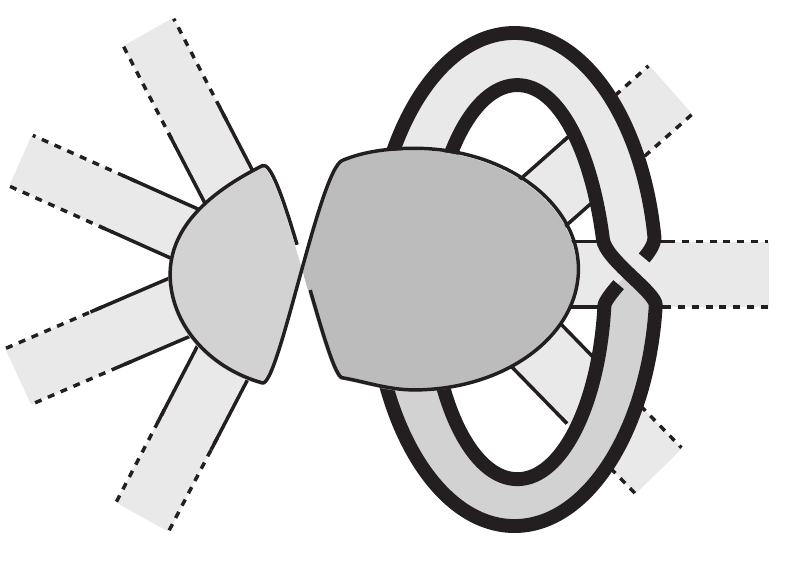} &\includegraphics[scale=.25]{ch4_38.pdf}
\\ \hline
\end{tabular}
\caption{Contraction and partial dual of an edge $e$ (highlighted in bold) in a ribbon graph.}
\label{tablecontractrg}
\end{table}

A \emph{quasi-tree} of a ribbon graph $G$ is a subgraph $(V(G),E')$, where $E'\subseteq E(G)$, that has a single boundary component for every component of $G$.
Note that each component of a quasi-tree of $G$, when viewed as a cellularly-embedded graph, has a single face.
In~\cite{delta-twist}, Chun, Moffatt, Noble, and Rueckriemen proved the following theorem, which is a restatement of a result by Bouchet~\cite{ab2}.

\begin{theorem}
\label{deltagraph}
Let $G$ be a ribbon graph with edge set $E$ and quasi-tree collection $\mathcal{Q}$.
Then $(E,\mathcal{Q})$ is a delta-matroid.
\end{theorem}

If $G$ is a ribbon graph we denote its associated delta-matroid by $D(G)$. Any delta-matroid arising in this way is called \emph{ribbon-graphic}.
Deviating slightly from standard practice, we say that a vertex $v$ of a connected graph is a \emph{cut-vertex} if there is a partition of the edges of the graph into two non-empty sets, so that $v$ is the only vertex incident with edges belonging to both sets of the partition. In contrast with the standard definition of $2$-connectivity, in a graph with at least two edges, any vertex incident with a loop is a cut-vertex.
A graph is \emph{$2$-connected} if it has a single connected component and has no cut-vertex. The point of our definition of $2$-connectivity is that a graph is $2$-connected if and only if its cycle matroid is connected.

From any ribbon graph $G$, we can derive an (abstract) graph, which we call the \emph{underlying abstract graph}, with a vertex corresponding to each vertex of $G$ and an edge corresponding to each edge of $G$, with incidences between edges and vertices if the corresponding vertex and edge intersect in $G$.
A \emph{cut-vertex} of a connected ribbon graph $G$ is any vertex $v$ that is a cut-vertex of the underlying abstract graph.
If $v$ is a cut-vertex of $G$, with $P$ and $Q$ being two ribbon subgraphs that intersect in $v$, such that neither $E(P)$ nor $E(Q)$ is empty and $E(P)\dot\cup E(Q)=E(G)$, then we say that $G=P\oplus Q$.
In this case, knowledge of $P$ and $Q$ gives complete knowledge of the underlying abstract graph of $G$, but does not give complete knowledge of $G$.
For example, suppose that $E(P)$ and $E(Q)$ are loops $p$ and $q$, respectively.
Then $G$ depends on the order in the order in which $p$ and $q$ are met when traveling around the boundary of the vertex $v$ and whether or not they are orientable.
Suppose that both $p$ and $q$ are orientable loops. If they are met in
the order $p,p,q,q$ when traveling around the boundary of $v$, then $G$ has three boundary components, whereas if they are met in the order $p,q,p,q$, then $G$ has one boundary component.
In the first case, $D(G)$ is disconnected, but in the second case
it is connected.
Because of this distinction, the two possible ribbon graphs have different connectivities, which we now define precisely.

Let $G$ be a ribbon graph.
We say that $G$ is \emph{connected} if it consists of a single connected component.
Two cycles $C_1$ and $C_2$ in $G$ are said to be \emph{interlaced} if there is  a vertex $v$ such that $V(C_1)\cap V(C_2)=\{v\}$, and  $C_1$ and $C_2$ are met in the cyclic order $C_1\,C_2\,C_1\,C_2$ when traveling around the boundary of the vertex $v$.
We say that  $G$ is the \emph{join} of $P$ and $Q$, written $G=P\vee Q$, if $G=P \oplus Q$ and no cycle in $P$ is interlaced with a cycle in $Q$.
In other words, $G$ can be obtained as follows: choose an arc on a vertex of $P$ and an arc on a vertex of $Q$ such that neither arc intersects an edge, then identify the two arcs merging the two vertices on which they lie into a single vertex of $G$.
The join  is also known as the ``one-point join,'' the ``map amalgamation,'' and the ``connected sum'' in the literature.
A ribbon graph is \emph{$2$-connected} exactly when it is connected and it is not the join of any pair of its subgraphs.
We refer the reader to \cite{Mo5,Mof11c} for a fuller discussion of separability for ribbon graphs.

The following results from~\cite[Proposition~5.21, Proposition~5.3, and Corollary~5.14]{delta-ribbon} provide the tools we need to reformulate our delta-matroid results as ribbon graph results.

\begin{proposition}
\label{coniscon}
Let $G$ be a ribbon graph.
Then
\begin{enumerate}[label=(\roman*)]
\item $D(G)$ is connected if and only if $G$ is $2$-connected;
\item $D(G)$ is even if and only if $G$ is orientable; and
\item for any edge $e$ of $G$, $D(G/e)=D(G)/e$ and $D(G\ba e)=D(G)\ba e$.
\end{enumerate}
\end{proposition}

We obtain the following corollaries of Theorem~\ref{chainmm} and Theorem~\ref{splitmm} for ribbon graphs.

\begin{corollary}\label{cor:chainorientribb}
Let $G$ be a $2$-connected orientable ribbon graph.
If $e\in E(G)$, then $G\ba e$ or $G/e$ is $2$-connected.
\end{corollary}

\begin{corollary}\label{cor:orientsplitter}
Let $G$ be a $2$-connected orientable ribbon graph with a $2$-connected minor $H$.
If $e\in E(G)-E(H)$, then $G\ba e$ or $G/e$ is $2$-connected with $H$ as a minor.
\end{corollary}
Unfortunately it is not possible to extend Corollary~\ref{cor:orientsplitter} to the class of all ribbon graphs, as the following example illustrates. Let $G$ be the ribbon graph formed by taking a planar embedding of the graph with two vertices and three parallel edges joining the two vertices, and giving a half-twist to one of the edges. Let $e$ denote the edge with a half-twist and let $a$, $b$ denote the other two edges. Then $G$ is $2$-connected with the $2$-connected minor $G/b\ba e$ comprising one vertex with an orientable loop attached. However $G/b$ is not $2$-connected. On the other hand $G\ba b$ is $2$-connected but does not contain $G/b\ba e$ as a minor.

However it is possible to exploit results of Brijder and Hoogeboom to establish a different splitter theorem for all ribbon graphs. We need to define three operations on delta-matroids and ribbon graphs. Bouchet introduced the twisting operation in~\cite{ab1}. Let $D=(E,\mathcal F)$ be a delta-matroid and let $A\subseteq E$. Then $D*A$ is the delta-matroid with ground set $E$ and collection of feasible sets $\{F\btu A:F\in \mathcal F\}$. It is easy to show that $D*A$ is indeed a delta-matroid. The analogous operation in ribbon graphs is the more complex operation of partial duality introduced by Chmutov in~\cite{Chmutov}.
For the purposes of this paper it is sufficient to define this operation by illustrating in Table~\ref{tablecontractrg} how to form $G*e$ for each type of edge $e$.
If $e_1$ and $e_2$ are edges of a ribbon graph $G$ then $(G*e_1)*e_2 = (G*e_2)*e_1 $, and so for $A=\{a_1, \ldots , a_n\}\subseteq E(G)$ we can define the {\em partial dual} of $G$ by $A$,  as $D*A= D*a_1*\cdots  *a_n$.
For more information see~\cite{Chmutov,EMMbook}. It is shown in~\cite{delta-ribbon} that these operations are compatible in the sense that if $G$ is a ribbon graph, then $D(G*A)=D(G)*A$.

Following  Brijder and Hoogeboom~\cite{BH11}, let  $D=(E,\mathcal{F})$ be a set system and  $e\in E$.
Then $D+e$ is defined to be the set system $(E,\mathcal{F}')$ where
$ \mathcal{F}'= \mathcal{F} \triangle \{ F\cup e : F\in \mathcal{F} \text{ and } e\notin F    \} $.
If $e_1, e_2 \in E$ then $(D+e_1)+e_2 = (D+e_2)+e_1 $, and so  for $A=\{a_1, \ldots , a_n\}\subseteq E$ we can define the {\em loop complementation} of $D$ by $A$,  as $D+A= D+a_1+\cdots  + a_n$.
Note that the set of delta-matroids is not closed under loop complementation.
A delta-matroid is said to be {\em vf-safe} if the application of any sequence of twists and loop complementations always results in a delta-matroid. The class of
vf-safe delta-matroids is known to be minor closed and strictly contains the class of ribbon-graphic delta-matroids (see~\cite{BH13}).
For a ribbon graph $G$ and set of edges $A$, let $G+A$ denote the ribbon graph formed by applying a half-twist to every edge in $A$. It is shown in~\cite{delta-twist} that loop-complementation and applying a half-twist are compatible operations, in the sense that $D(G)+A = D(G+A)$. For a delta-matroid $D$ (respectively ribbon graph $G$), we define $D\bar * A= D+A*A+A$ (respectively $G\bar * A= G+A*A+A$).

Brijder and Hoogeboom have recently shown in~\cite{BH14} that there is a natural correspondence between vf-safe delta-matroids and tight 3-matroids as follows. Let $E$ be a finite set and let $E_0=E$, $E_1=\{e':e\in E\}$ and $E_2=\{e'':e\in E\}$.
Let $U=E_0\cup E_1 \cup E_2$ and $\Omega=\{\{e,e',e''\}:e\in E\}$.
There is a natural projection $\pi$ mapping transversals of
$\Omega$ to subsets of $E$.
\begin{theorem}[Brijder and Hoogeboom]\label{thm:BH}
Using the notation from above, there is a one-to-one correspondence between vf-safe delta-matroids with ground set $E$, and tight $3$-matroids with ground set $U$ and set $\Omega$ of skew classes, given by the following map.
The vf-safe delta-matroid $D$ is mapped to the tight $3$-matroid $Q_3(D)$ in which
a transversal $B$ is a basis of $Q_3(D)$ if and only if $\pi(B \cap E_1)$ is a feasible set of $D \bar * \pi(B\cap E_2)$.
The inverse map takes a tight $3$-matroid $Q$ to a vf-safe delta-matroid $D(Q)$ in which $F$ is feasible if and only if there is a basis $B$ of $Q$ such that $B \subseteq E_0 \cup E_1$ and $\pi(B\cap E_1)=F$.
\end{theorem}

Moreover, as shown in~\cite{BH14}, minor operations are preserved by this correspondence in the following sense. Let $e \in E$. Then
\begin{equation}
Q_3(D\ba e) = Q_3(D)|e, \quad
Q_3(D/e) = Q_3(D)|e', \quad
Q_3(D+e/e) = Q_3(D)|e''. \label{eq:minor}
\end{equation}
The third equation above suggests a third minor operation in vf-safe delta-matroids and, as a consequence, ribbon-graphs.
We call the operation of taking a loop complementation with respect to $e$ followed immediately by contracting $e$ to be the \emph{twist-contraction} of $e$.
It is not difficult to show that in both ribbon graphs and delta-matroids, the order in which a set of deletions, contractions and twist-contractions is applied does not affect the result.
If $D$ is a vf-safe delta-matroid, then we say that $D'$ is a $3$-minor of $D$ if $D'$ may be obtained from $D$ by a sequence of deletions, contractions and twist-contractions. Similarly we say that a ribbon graph $H$ is a $3$-minor of a ribbon graph $G$ if $H$ may be obtained from $G$ by a sequence of deletions of edges, deletions of vertices, contractions of edges and twist-contractions of edges.

In order to translate results from the setting of tight $3$-matroids to vf-safe delta-matroids, we need one final result.
\begin{proposition}\label{conn3}
Let $D=(E,\mathcal F)$ be a vf-safe delta-matroid. Then $D$ is connected if and only if $Q_3(D)$ is connected.
\end{proposition}
\begin{proof}
It is clear from the form of the map taking a tight $3$-matroid to a vf-safe delta-matroid that if $Q_3(D)$ is disconnected, then so is $D$. We now prove the converse. We claim
that if $X$ is separator of $D$, then it is also a separator of both $D+A$, $D*A$ and $D\bar *A$ for any subset $A$ of $E(D)$. It is simple to verify this claim in the case that $A$ comprises a single element and then the claim follows using an easy induction.

We keep the notation used above in the construction of $Q_3(D)$, in particular $E_0$, $E_1$, $E_2$ and $\pi$.
Suppose that $X$ is a proper separator of $D$.
Thus $D=D_1 \oplus D_2$, where $E(D_1)=X$ and $E(D_2)=E-X$.
Let $U$ denote the ground set of $Q_3(D)$ and $\Omega$ the partition of $U$ into skew classes. Recall that each skew class corresponds to an element of $E$.
Let $Y$ denote the union of all the skew classes of $Q_3(D)$ corresponding to elements of $X$.
The condition that $B$ is a basis of $Q_3(D)$ is equivalent to saying that $\pi(B\cap E_1)$ is a feasible set of $D\bar * \pi(B\cap E_2)$.
This in turn is equivalent to saying that $\pi(B\cap E_1) \cap X$ is a feasible set of $D\bar * \pi(B\cap E_2)|X$ and
$\pi(B\cap E_1) \cap (E-X)$ is a feasible set of $D\bar * \pi(B\cap E_2)|(E-X)$. Now this holds if and only if $\pi(B\cap Y \cap E_1)$ is a feasible set of $D_1\bar * \pi(B\cap Y \cap E_2)$
and $\pi(B\cap (U-Y) \cap E_1)$ is a feasible set of $D_2\bar * \pi(B\cap (U-Y) \cap E_2)$. Finally this is equivalent to saying that $B\cap Y$ is a basis of $Q_3(D_1)$ and $B\cap (U-Y)$ is a basis of $Q_3(D_2)$.
Thus $Y$ is a proper separator of $Q_3(D)$. \qed
\end{proof}

Combining Theorem~\ref{chainmm} and Proposition~\ref{conn3} with Theorem~\ref{thm:BH} and~\eqref{eq:minor}, we obtain the following.
\begin{corollary}
Let $D$ be a connected vf-safe delta-matroid. If $e\in E(D)$, then at least two of $D\ba e$, $D/e$ and $D+e/e$ are connected.
\end{corollary}

\begin{corollary}
Let $G$ be a $2$-connected ribbon graph. If $e\in E(G)$, then at least two of $G\ba e$, $G/e$ and $G+e/e$ are $2$-connected.
\end{corollary}

It follows immediately that we can drop the orientability condition from Corollary~\ref{cor:chainorientribb}.
\begin{corollary}
Let $G$ be a $2$-connected ribbon graph. If $e\in E(G)$, then $G\ba e$ or $G/e$ is $2$-connected.
\end{corollary}

Finally, by combining Theorem~\ref{chainmm} and Proposition~\ref{conn3} with Theorem~\ref{thm:BH} and~\eqref{eq:minor}, we obtain the following.

\begin{corollary}
\label{cor:splitterdelta}
Let $D$ be a connected delta-matroid with a connected $3$-minor $D'$.
If $e\in E(D)-E(D')$, then $D\ba e$, $D/e$ or $D+e/e$ is connected with $D'$ as a $3$-minor.
\end{corollary}

\begin{corollary}\label{cor:graphsplitter}
Let $G$ be a $2$-connected ribbon graph with a $2$-connected $3$-minor $H$.
If $e\in E(G)-E(H)$, then $G\ba e$, $G/e$ or $G+e/e$ is $2$-connected with $H$ as a $3$-minor.
\end{corollary}

\section*{Acknowledgements}
We would like to thank Iain Moffatt for helpful discussions and for his assistance with the figures, and the anonymous referees for a careful reading and several suggestions that improved the exposition, in particular, for recommending that we include Lemma~\ref{lem:ref} and providing the proof.

\bibliographystyle{elsarticle-num}

\end{document}